\documentclass[11pt]{article}
\usepackage{graphicx}
\usepackage{amssymb}
\usepackage{amsmath}
\usepackage{amsthm}
\newtheorem{definition}{Definition}[section]
\newtheorem{lemma}{Lemma}[section]
\newtheorem{theorem}{Theorem}[section]
\newtheorem{proposition}{Proposition}[section]

{\end{description}}
\newenvironment{example}{\begin{description} \item[\bf Example:] }%
{\end{description}}

{\hfill $\bullet$ \\}

\newcommand{\be}{\begin{equation}}
\newcommand{\ee}{\end{equation}}

\newcommand{\diag}{\begin{smallmatrix}\vspace{-0.5ex}\textrm{\normalsize diag}\\\vspace{-0.8ex}j=1,\ldots,n\end{smallmatrix}}

\begin{document}
\title{Multigrid methods for Toeplitz linear systems with different size reduction}
\author{Marco Donatelli, Stefano Serra-Capizzano
       \thanks{Dipartimento di Fisica e
       Matematica, Universit\`a dell'Insubria,
       Via Valleggio 11, 22100 Como (ITALY).
       Email:\ \{marco.donatelli,stefano.serrac\}@uninsubria.it; serra@mail.dm.unipi.it}
       , and Debora Sesana
       \thanks{Dipartimento di Scienze Economiche e Metodi Quantitativi, 
       Universit\`a degli Studi del Piemonte Orientale - Amedeo Avogadro,
       Via Perrone 18, 28100 Novara (ITALY).
       Email:\ debora.sesana@eco.unipmn.it}
       }
\maketitle
\date{}

\begin{abstract}
Starting from the spectral analysis of $g$-circulant matrices, we
consider a new multigrid method for circulant and Toeplitz matrices with given
generating function. We assume that the size $n$ of the coefficient matrix is
divisible by $g\ge 2$ such that at the lower level the system is reduced
to one of size $n/g$ by employing $g$-circulant based projectors. We
perform a rigorous two-grid convergence analysis in the circulant case
and we extend experimentally the results to the Toeplitz setting,
by employing structure preserving projectors. The optimality
of the proposed two-grid method and of the multigrid method is proved, when the number $\theta \in \mathbb{N}$
of recursive calls is such that $1 < \theta < g$. The previous analysis is used
also to overcome some pathological cases, in which the generating
function has zeros located at ``mirror points'' and the standard
two-grid method with $g=2$ is not optimal.
The numerical experiments show the correctness and
applicability of the proposed ideas both for circulant and Toeplitz matrices.
\end{abstract}
\\
{\bf AMS SC:} 65N55, 65F10, 65F15.

\section{Introduction}
\label{intro}
In the last 20 years multigrid methods have gained a remarkable
reputation as fast solvers for structured matrices associated to shift
invariant operators where the size $n$ is large and the system shows
a conditioning growing polynomially with $n$ (see
\cite{FS1,Smulti,FS2,CCS,HS,SCC,ADS,mcirco,H} and the references therein).
Under suitable mild assumptions, the considered techniques are optimal showing
linear or almost linear ($O(n\log n)$ arithmetic operations as the
celebrated fast Fourier transform (FFT)) complexity for reaching the
solution within a preassigned accuracy and a convergence rate
independent of the size $n$ of the involved system. These excellent
features carry over the multilevel setting and mimic very well
those already known in the context of elliptic ordinary and partial
differential equations (see \cite{Hack,RStub,Trot,Sun} and the
references therein). In particular, if the underlying structures are also sparse
as in the multilevel banded case, then the cost of solving the involved
linear system is proportional to the order of the coefficient matrix with constant
depending linearly on the bandwidths at each level.
We mention that the cost of direct methods is $O(n\log n)$ operations
in the case of trigonometric matrix algebras (circulant, $\tau$, \dots)
and it is $O(n^{\frac{3d-1}{d}})$ for $d$-level Toeplitz matrices (see \cite{KCM}).
Concerning multilevel Toeplitz structures, superfast methods represents a good alternative,
even if the algorithmic error has to be controlled, with a cost of $O(n^{\frac{3d-2}{d}}\log^2(n))$:
the cost is really competitive for $d=1$, while the deterioration is evident for $d>1$ since
it is nontrivial to exploit the structure at the inner levels (see \cite{barel} and references therein).
Moreover, in the last case the most popular preconditioning strategies by matrix algebra
can be far from being optimal in the multidimensional case (see \cite{negablo}).
On the other hand, multigrid method are optimal also for polynomially ill-conditioned
multidimensional problems and they can be extended to the case of low rank corrections
of the consider structured matrices, allowing to deal also with
the modified Strang preconditioner widely used in the literature
(see \cite{CNrev} and the references therein).

The main novelty contained in the works from the structured matrices
literature is the use of the symbol. Indeed, starting from the
initial proposal in \cite{FS1}, we know that the convergence
analysis of the two-grid and V-cycle can be handled in a compact and
elegant manner by studying few analytical properties of the symbol
(so the study does not involve the entries of the matrix and, more
importantly, the size $n$ of the system). Already in the two-grid
convergence analysis, it is evident that the optimality can be
reached only if the symbol $f$ has a finite number of zeros of
finite order and not located at mirror points: more explicitly, if
$x_0$ is a zero of $f$ then $f(x_0+\pi)$ must be greater than zero. Here we show that the
second requirement is not essential since it depends on the choice
of projecting the original matrix of size $n$ into a new one of size
$n/2$. The latter is not compulsory so that, by choosing a different
size reduction from $n$ to $n/g$ and $g>2$, we can overcome the
pathology induced by the mirror points.
Other approaches for dealing with such pathologies were proposed in \cite{CCS,HS}.

In this paper we propose a new multigrid method where the fine problem of size $n$
is projected to a coarser problem of size $n/g$, $g=2,3,\dots$.
We perform a two-grid analysis using the ideas in \cite{mcirco}
for circulant structures and by exploiting the
spectral analysis of $g$-circulant matrices already performed in
\cite{NSS}.
As shown in \cite{D}, such two-grid analysis is an algebraic generalization of the classical local
Fourier analysis and it allows one to apply the results obtained for circulant
matrices also to Toeplitz matrices. A feature of our multigrid is that
the coarse problem of size $n/g$ with $g>2$
allows to obtain optimal multigrid methods with $g-1$ recursive calls:
it is enough to perform the analysis of the arithmetic computational cost related to the size reduction $n/g$
between two consecutive levels and to invoke the results in \cite{Trot}.
A further property of the proposed multigrid is that the pathologies induced by the
mirror points are bypassed as previously described. Our proposal could be
extended to the multilevel case by tensor product arguments considering
the increasing of the number of  ``mirror points''. Moreover a V-cycle
convergence analysis could be performed by following the steps in \cite{ADS,AD}
as a model. A rigorous study in this directions will be the subject of future research.

The paper is organized as follows. In Section~\ref{sec:circ} we set
the problem by recalling the main features of the circulant, Toeplitz and $g$-circulant
matrices. In Section~\ref{multi:circ} we report definitions and
classical convergence results concerning two-grid and multigrid iterations from \cite{RStub,Trot}.
In Section~\ref{cutting:circ} we define our grid-transfer operators and
we study the properties of the coarse matrix obtained by the Galerkin approach.
Section~\ref{proof:circ} is devoted to the proof of convergence of
our multigrid when applied to circulant matrices and to briefly discuss the
pathologies, that are eliminated by our algorithmic proposal.
Section \ref{sec:num} is
concerned with numerical experiments regarding circulant and Toeplitz matrices (e.g.
ill-conditioned linear systems coming from approximated differential
and integral problems). Section \ref{sec:final} is devoted to
conclusions and to sketch future lines of research.

\section{Circulant matrices and other related structures}
\label{sec:circ}

Let $f$ be a trigonometric polynomial defined over the set $Q=[0,2\pi)$
and having degree $c\geq 0$, i.e., $f(u)=\sum_{k=-c}^c a_k e^{iku}$, $i^{2}=-1$. From the Fourier coefficients of $f$,
that is
\begin{eqnarray}\label{coeff}
  a_{j}=\frac{1}{2\pi}\int_{Q}f(u)e^{-iju}du,\qquad j\in\mathbb{Z},
\end{eqnarray}
one can build the circulant matrix $C_n(f)=\left[a_{(r-s)\ {\rm
mod}\, n}\right]_{r,s=0}^{n-1}$. For example, let $f(u)=3-2e^{iu}+e^{-2iu}$. 
The degree of $f$ is $c=2$ and we have
$a_{0}=3$, $a_{1}=-2$ and $a_{-2}=1$; if we take $n>(2c-1)$, see the discussion below,
since $a_{(-2)\ {\rm mod}\, n}=a_{3}$ the circulant matrix $C_{5}(f)$ is given by
\begin{eqnarray*}
  C_{5}(f)=\left[\begin{array}{ccccc}
     3 &  0 &  1 &  0 & -2 \\
    -2 &  3 &  0 &  1 &  0 \\
     0 & -2 &  3 &  0 &  1 \\
     1 &  0 & -2 &  3 &  0 \\
     0 &  1 &  0 & -2 &  3
  \end{array}\right].
\end{eqnarray*}

It is clear that the Fourier
coefficient $a_{j}$ equals zero if the condition $|j|\leq c$ is
violated. The matrix $C_{n}(f)$ is said to be the circulant matrix
of order $n$ generated by $f$, and can be written as
$C_{n}(f)=\sum_{|j|\leq c}a_{j}Z_{n}^{j}$, where the matrix
\begin{eqnarray*}
  Z_{n}=\left[\begin{array}{cccc}
         0 & \cdots & 0 & 1\\
         1 &        &   & 0\\
           & \ddots &   & \vdots\\
         0 &        & 1 & 0
  \end{array}\right]
\end{eqnarray*}
is the cyclic permutation Toeplitz matrix.
In addition, if $F_{n}$ denotes the Fourier matrix of size $n$, i.e.
\begin{eqnarray}\label{V}
 F_{n} =\frac{1}{\sqrt{n}}\left[e^{-\frac{2\pi ijk}{n}}\right]_{j,k=0}^{n-1},
\end{eqnarray}
then it is well known (see e.g. \cite{Da}) that
\begin{eqnarray}\label{cn}
  C_{n}(f)=F_{n}\Delta_{n}(f)F_{n}^{H},
\end{eqnarray}
where
\begin{eqnarray}\label{deltan}
  \Delta_{n}(f)&=&\diag f\left(x_{j}^{(n)}\right),\qquad x_{j}^{(n)}=\frac{2\pi j}{n},\\
  \nonumber&=&{\rm diag}(\sqrt{n}F_{n}^{H}a),
  \qquad\qquad a=\left[a_{0},a_{1},\ldots,a_{n-1}\right]^{T},
\end{eqnarray}
$a$ being the first column of the matrix $C_{n}(f)$.

Under the assumption that $c\le
\left\lfloor(n-1)/2\right\rfloor$, the matrix $C_n(f)$ is
the Strang or natural circulant preconditioner of the corresponding
Toeplitz matrix $T_n(f)=\left[a_{(r-s)}\right]_{r,s=0}^{n-1}$
(see~\cite{CNrev} and the references therein). We observe that the
above-mentioned assumption $c\le \left\lfloor(n-1)/2\right\rfloor$ is fulfilled
at least definitely, since each $c$ is a fixed constant and $n$ is the matrix order:
in actuality, in real applications it is
natural to suppose that $n$ is large if we assume that $C_n(f)$
comes from an approximation process of an infinite-dimensional
problem. However if the symbol $f$ has a zero at zero (this happens in the case of
approximation of differential operators), then $C_n(f)$ is singular and it is usually replaced by a rank-one correction
that forces invertibility: the latter is called in the relevant literature
modified Strang preconditioner.

We end this section with the definition of $g$-circulant matrix.
A matrix $C_{n,g}$ of size $n$ is called $g$-circulant if its entries
obey the rule $C_{n,g}=\left[a_{(r-g s)\ {\rm mod}\,n}\right]_{r,s=0}^{n-1}$:
for an introduction and for the algebraic
properties of such matrices refer to Section 5.1 of the classical
book by Davis \cite{Da}, while new additional results can be found in
\cite{NSS,trench} and references therein. For instance if $n=5$ and $g=3$ then we
have
\begin{eqnarray*}
  C_{n,g}=\left[\begin{array}{ccccc}
    a_{0} & a_{2} & a_{4} & a_{1} & a_{3} \\
    a_{1} & a_{3} & a_{0} & a_{2} & a_{4} \\
    a_{2} & a_{4} & a_{1} & a_{3} & a_{0} \\
    a_{3} & a_{0} & a_{2} & a_{4} & a_{1} \\
    a_{4} & a_{1} & a_{3} & a_{0} & a_{2}
  \end{array}\right].
\end{eqnarray*}

Also in this case, as in ordinary circulant setting, the coefficients 
$a_{j}$, $j\in\mathbb{Z}$, could arise from a given symbol $f$ (see (\ref{coeff})).
For instance, with $f(u)=3-2e^{iu}+e^{-2iu}$, we find $a_{0}=3$, $a_{1}=-2$, and $a_{-2}=a_3=1$
so that
\begin{eqnarray*}
  C_{5,3}(f)=\left[\begin{array}{ccccc}
    3 & 0 & 0 & -2 & 1 \\
    -2 & 1 & 3 & 0 & 0 \\
    0 & 0 & -2 & 1 & 3 \\
    1 & 3 & 0 & 0 & -2 \\
    0 & -2 & 1 & 3 & 0
  \end{array}\right].
\end{eqnarray*}


\section{Two-grid and Multigrid methods}\label{multi:circ}

Let $A_{n}\in\mathbb{C}^{n\times n}$, and
$x_{n},\,b_{n}\in\mathbb{C}^{n}$. Let
$p_{n}^{k}\in\mathbb{C}^{n\times k}$, $k<n$, a given full-rank
matrix and let us consider a class of iterative methods of the form
\begin{eqnarray}\label{prepost}
  x_{n}^{(j+1)}=V_{n}x_{n}^{(j)}+\tilde{b}_{n}:=\mathcal{V}(x_{n}^{(j)},\tilde{b}_{n}),
\end{eqnarray}
where $A_{n}=W_{n}-N_{n}$, $W_{n}$ nonsingular matrix, $V_{n}:=I_{n}-W_{n}^{-1}A_{n}\in\mathbb{C}^{n\times n}$,
and $\tilde{b}_{n}:=W_{n}^{-1}b_{n}\in\mathbb{C}^{n}$.
A Two-Grid Method (TGM) is defined by the following algorithm:
\begin{center}
\begin{tabular}{l}
\vspace{1ex}
TGM$(V_{n,\rm{pre}}^{\nu_{\rm{pre}}},V_{n,\rm{post}}^{\nu_{\rm{post}}},p_{n}^{k})(x_{n}^{(j)})$
\vspace{1ex}\\
\hline
  0. $\tilde{x}_{n}=\mathcal{V}_{n,\rm{pre}}^{\nu_{\rm{pre}}}(x_{n}^{(j)},\tilde{b}_{n,\rm{pre}})$\\
  1. $d_{n}=A_{n}\tilde{x}_{n}-b_{n}$\\
  2. $d_{k}=(p_{n}^{k})^{H}d_{n}$\\
  3. $A_{k}=(p_{n}^{k})^{H}A_{n}p_{n}^{k}$\\
  4. Solve $A_{k}y=d_{k}$\\
  5. $\hat{x}_{n}=\tilde{x}_{n}-p_{n}^{k}y$\\
  6. $x_{n}^{(j+1)}=\mathcal{V}_{n,\rm{post}}^{\nu_{\rm{post}}}(\hat{x}_{n},\tilde{b}_{n,\rm{post}})$
\end{tabular}
\end{center}

Steps $1.\rightarrow5.$ define the ``coarse grid correction'' that
depends on the projector operator $p_{n}^{k}$, while Step $0.$ and
Step $6.$ consist, respectively, in applying $\nu_{\rm{pre}}$ times
and $\nu_{\rm{post}}$ times a ``pre-smoothing iteration'' and a
``post-smoothing iteration'' of the generic form given in
$(\ref{prepost})$.
The global iteration matrix of the TGM is then given by
\begin{eqnarray*}
  {\rm TGM}(V_{n,\rm{pre}}^{\nu_{\rm{pre}}},V_{n,\rm{post}}^{\nu_{\rm{post}}},p_{n}^{k})=
  V_{n,\rm{post}}^{\nu_{\rm{post}}}
  \left[I_{n}-p_{n}^{k}\left((p_{n}^{k})^{H}
  A_{n}p_{n}^{k}\right)^{-1}(p_{n}^{k})^{H}A_{n}\right]V_{n,\rm{pre}}^{\nu_{\rm{pre}}}.
\end{eqnarray*}

If $k$ is large, the numerical solution to the linear system at the Step $4.$ could
be computationally expensive. In such case a multigrid procedure is adopted.
Fix $0<m<n$, the sequence $0<n_m<n_{m-1}<\dots<n_1<n_0=n$ and the full-rank matrices
$p_{n_{i-1}}^{n_i} \in \mathbb{C}^{{n_{i-1}}\times{n_i}}$, for $i=1,\dots,m$.
The multigrid method produces the sequence $\{x_n^{(k)}\}_{k\in\mathbb{N}}$ defined by
$x_{n}^{(j+1)} = {\rm MGM}(V_{n,\rm{pre}}^{\nu_{\rm{pre}}},V_{n,\rm{post}}^{\nu_{\rm{post}}},
p_{n}^{n_1},A_n,b_n,\theta,0)(x_{n}^{(j)})$
with the function MGM defined recursively as follows:
\begin{center}
\begin{tabular}{l}
\vspace{1ex}
$x_{n_i}^{(j+1)} = {\rm MGM}(V_{n_i,\rm{pre}}^{\nu_{\rm{pre}}},V_{n_i,\rm{post}}^{\nu_{\rm{post}}},p_{n_i}^{n_{i+1}},A_{n_i},b_{n_i},\theta,i)(x_{n_i}^{(j)})$
\vspace{1ex}\\
\hline
If $i=m$ then Solve $A_{n_i}x_{n_i}^{(j+1)}=b_{n_i}$\\
Else\\
  \hspace{0.4cm} 0. $\tilde{x}_{n_i}=\mathcal{V}_{n_i,\rm{pre}}^{\nu_{\rm{pre}}}(x_{n_i}^{(j)},\tilde{b}_{n_i,\rm{pre}})$\\
  \hspace{0.4cm} 1. $d_{n_i}=A_{n_i}\tilde{x}_{n_i}-b_{n_i}$\\
  \hspace{0.4cm} 2. $d_{n_{i+1}}=(p_{n_i}^{n_{i+1}})^{H}d_{n_i}$\\
  \hspace{0.4cm} 3. $A_{n_{i+1}}=(p_{n_i}^{n_{i+1}})^{H}A_{n_i}p_{n_i}^{n_{i+1}}$\\
  \hspace{0.4cm} 4. $x_{n_{i+1}}^{(j+1)} = 0$\\
  \hspace{0.7cm} for $s=1$ to $\theta$\\
  \hspace{1.2cm}  $x_{n_{i+1}}^{(j+1)} = {\rm MGM}(V_{n_{i+1},\rm{pre}}^{\nu_{\rm{pre}}},V_{n_{i+1},\rm{post}}^{\nu_{\rm{post}}},p_{n_{i+1}}^{n_{i+2}},A_{n_{i+1}},d_{n_{i+1}},\theta,i+1)(x_{n_i}^{(j+1)})$\\
  \hspace{0.4cm} 5. $\hat{x}_{n_i}=\tilde{x}_{n_i}-p_{n_i}^{n_{i+1}}x_{n_{i+1}}^{(j+1)}$\\
  \hspace{0.4cm} 6. $x_{n_i}^{(j+1)}=\mathcal{V}_{n_i,\rm{post}}^{\nu_{\rm{post}}}(\hat{x}_{n_i},\tilde{b}_{n_i,\rm{post}})$
\end{tabular}
\end{center}
The choices $\theta=1$ and $\theta=2$ correspond to the well-known V-cycle and
W-cycle, respectively.

In the present paper, we are interested in proposing such kind of
techniques in the case where $A_{n}$ is a Toeplitz matrix.
However, for a theoretical analysis, we consider circulant matrices
according to the local Fourier analysis for classical multigrid
methods (see \cite{D}).
Even if we treat in detail the circulant case, in the spirit of the
paper \cite{ADS}, the same ideas can be plainly translated to other
matrix algebras associated to (fast) trigonometric transforms.
First we recall some convergence results from the theory of the algebraic
multigrid method given in \cite{RStub}.

By $\|\cdot\|_{2}$ we denote the Euclidean norm on $\mathbb{C}^{n}$
and the associated induced matrix norm over $\mathbb{C}^{n\times
n}$. If $X$ is positive definite,
$\|\cdot\|_{X}=\|X^{1/2}\cdot\|_{2}$ denotes the Euclidean norm
weighted by $X$ on $\mathbb{C}^{n}$ and the associated induced
matrix norm. Finally, if $X$ and $Y$ are Hermitian matrices, then
the notation $X\leq Y$ means that $Y-X$ is nonnegative definite.
In the following we use some functional norms: more precisely the
usual $L^{\infty}$ norm $\|\cdot\|_{\infty}$ defined as $\|f\|_{\infty}=\sup_{x\in Q}|f(x)|$,
and the weighted $L^{1}$ norm $\|\cdot\|_{1}$ defined as
$\|f\|_{1}=\frac{1}{2\pi}\int_{Q}|f(x)|dx$ (according to the Haar measure).

\begin{theorem}[\cite{RStub}]\label{teoconv}
Let $A_{n}$ be a positive definite matrix of size $n$ and let $V_{n}$ be defined as in the {\rm TGM} algorithm.
Suppose that there exists $\alpha_{\rm{post}}>0$ independent of $n$ such that
\begin{eqnarray}\label{cond1}
  \|V_{n,\rm{post}}x_{n}\|_{A_{n}}^{2}\leq\|x_{n}\|_{A_{n}}^{2}-\alpha_{\rm{post}}\|x_{n}\|_{A_{n}D_{n}^{-1}A_{n}}^{2},\qquad \forall x_{n}\in\mathbb{C}^{n},
\end{eqnarray}
where $D_{n}$ is the main diagonal of $A_{n}$. Assume that there exists $\gamma>0$ independent of $n$ such that
\begin{eqnarray}\label{cond3}
  \min_{y\in\mathbb{C}^{k}}\|x_{n}-p_{n}^{k}y\|_{D_{n}}^{2}\leq \gamma\|x_{n}\|_{A_{n}}^{2},\qquad \forall x_{n}\in\mathbb{C}^{n}.
\end{eqnarray}
Then $\gamma\geq\alpha_{\rm{post}}$ and
\begin{eqnarray*}
  \|{\rm TGM}(I,V_{n,\rm{post}}^{\nu_{\rm{post}}},p_{n}^{k})\|_{A_{n}}\leq\sqrt{1-\alpha_{\rm{post}}/\gamma}.
\end{eqnarray*}
\end{theorem}
Conditions \eqref{cond1} and \eqref{cond3} are usually called as ``smoothing property'' and
``approximation property'', respectively.

We note that $\alpha_{\rm{post}}$ and $ \gamma $
are independent of $n$ and hence, if the assumptions of Theorem \ref{teoconv}
are satisfied, the resulting TGM is not only convergent but also
optimal. In other words, the number of iterations in order to reach a given accuracy $\epsilon$ can be
bounded from above by a constant independent of $n$ (possibly depending on the parameter $\epsilon$).

Of course, if the given method is complemented with a convergent pre-smoother, then by the same theorem
we get a faster convergence. In fact, it is known that for square matrices $A$ and $B$ the spectra of $AB$ and $BA$
coincide.

Therefore ${\rm TGM}(V_{n,\rm{pre}}^{\nu_{\rm{pre}}},V_{n,\rm{post}}^{\nu_{\rm{post}}},p_{n}^{k})$
and ${\rm TGM}(I,V_{n,\rm{pre}}^{\nu_{\rm{pre}}}V_{n,\rm{post}}^{\nu_{\rm{post}}},p_{n}^{k})$ have the same ei\-gen\-values
so that
\begin{eqnarray*}
\|{\rm TGM}(V_{n,\rm{pre}}^{\nu_{\rm{pre}}},V_{n,\rm{post}}^{\nu_{\rm{post}}},p_{n}^{k})\|_{A_{n}}=
\|{\rm TGM}(I,V_{n,\rm{pre}}^{\nu_{\rm{pre}}}V_{n,\rm{post}}^{\nu_{\rm{post}}},p_{n}^{k})\|_{A_{n}}
< \sqrt{1-\alpha_{\rm{post}}/\gamma},
\end{eqnarray*}
and hence the presence of a pre-smoother can only improve the convergence.

Concerning multigrid methods, in \cite{RStub} the V-cycle convergence is considered
with a result which could be seen as the analog of Theorem \ref{teoconv}.
For other bounds concerning the convergence rate of the V-cycle see \cite{Notay} and reference therein.
Regarding the convergence of the W-cycle, we point out that a rigorous TGM analysis is sufficient
for determining the optimality of the W-cycle (see \cite{Trot}).

\section{Projector operators for circulant matrices}\label{cutting:circ}

Let $A_{n}:=C_{n}(f)$ be a circulant matrix generated by a
trigonometric polynomial $f$. In order to provide a general method
for obtaining a projector operator from an arbitrary banded circulant
matrix $P_{n}$, for some bandwidth independent of $n$, we introduce
the operator $Z_{n,g}^{k}\in\mathbb{R}^{n\times k}$,
$k=\frac{n}{g}\in\mathbb{N}$, where
\begin{eqnarray}\label{Z}
  Z_{n,g}^{k}=[\delta_{i-gj}]_{i,j},\qquad \delta_{r}=
  \left\{\begin{array}{cll} 1 & \textrm{if $r\equiv 0\textrm{ (mod
$n$)}$,} & \qquad i=0,\ldots,n-1,\\
0 & \textrm{otherwise,}&\qquad j=0,\ldots,k-1.\end{array} \right.
\end{eqnarray}
The operator $Z_{n,g}^{k}$ represents a special link between the
space of the frequencies of size $n$ and the corresponding space of
frequencies of size $k$.

\begin{lemma}
Let $F_{n}$ be the Fourier matrix of size $n$ defined in $(\ref{V})$
and let $Z_{n,g}^{k}\in\mathbb{R}^{n\times k}$ be the
matrix defined in $(\ref{Z})$. If $k=\frac{n}{g}\in\mathbb{N}$ then
\begin{eqnarray}\label{decompos}
  F_{n}^{H}Z_{n,g}^{k}=\frac{1}{\sqrt{g}}I_{n,g}F_{k}^{H},
\end{eqnarray}
where $I_{n,g}\in \mathbb{R}^{n\times k}$ and
\begin{eqnarray*}
  I_{n,g}=\left.\left[\begin{array}{c}
  I_{k}\\
  \hline
  I_{k}\\
  \hline
  \vdots\\
  \hline
  I_{k}
  \end{array}\right]\right\}\textrm{$g$ times,}
\end{eqnarray*}
with $I_{k}$ being the identity matrix of size $k$.
\end{lemma}

This simple relation (see \cite[Lemma 3.3 and Remark 3.5]{NSS} for the details of the proof
and \cite{trench} for recent findings on these structures)
is the key step in defining an algebraic multigrid method, since it
allows us to obtain again a circulant matrix at the lower level.
Indeed, denoting by $\Delta_{n}$ the diagonal matrix obtained from
the eigenvalues of $A_{n}$ (see $(\ref{deltan})$), we infer that
$\Delta_{k}:=\frac{1}{g}I_{n,g}^T\Delta_{n}I_{n,g}$
is again a diagonal matrix. Therefore
\begin{eqnarray*}
  (Z_{n,g}^{k})^{H}A_{n}Z_{n,g}^{k}&=&(Z_{n,g}^{k})^{H}F_{n}\Delta_{n}F_{n}^{H}Z_{n,g}^{k}\\
  &=&\frac{1}{g}F_{k}I_{n,g}^T\Delta_{n}I_{n,g}F_{k}^{H}\\
  &=&F_{k}\Delta_{k}F_{k}^{H}\\
  &=&A_{k},
\end{eqnarray*}
where $A_{k}$ is a new circulant matrix. Consequently, starting from the matrix
$Z_{n,g}^{k}$, it is possible to define a generic projector
\begin{eqnarray}\label{procirc}
  p_{n,g}^{k}=P_{n}Z_{n,g}^{k},
\end{eqnarray}
where $P_{n}$ is a circulant matrix. Indeed $P_{n}^{H}A_{n}P_{n}$ is
a circulant matrix and then
$A_{k}=(p_{n,g}^{k})^{H}A_{n}p_{n,g}^{k}$ is again a circulant matrix. We
note that, since $k=\frac{n}{g}\in\mathbb{N}$, $n$ must be a
multiple of $g$. We are left to determine the conditions to be satisfied
by $P_{n}=C_{n}(p)$ (or better by its generating function $p$) in order to get
a projector which is effective in terms of convergence.

\begin{definition}
Given $x\in[0,2\pi)$, $g \in \mathbb{N}$, $g \geq 2$, the set of \emph{$g$-corners}
of $x$ is $\Omega_g(x)=\{y = x +\frac{2\pi k}{g}$ (mod $2\pi$) $\mid k=0,\ldots,
g-1\}$ and the set of \emph{$g$-mirror points} is $\mathcal{M}_g(x) = \Omega_g(x)\setminus\{x\}$.
\end{definition}

\noindent
\textbf{TGM conditions}
\hspace{0.05cm}
\emph{Let $A_{n}:=C_{n}(f)$ with $f$ nonnegative, trigonometric
polynomial and let $p_{n,g}^{k}=C_{n}(p)Z_{n,g}^{k}$ with
$p$ trigonometric polynomial. Assume that $f(x_{0})=0$ for
$x_{0}\in[0,2\pi)$, choose $p$ such that the following relations
\begin{eqnarray}
  \lim_{x\rightarrow x_{0}}\frac{p^{2}(y)}{f(x)}<\infty, \qquad & \forall \, y \in \mathcal{M}_g(x), \label{p2f1}\\
  \sum_{y \in \Omega_g(x)} p^2(y)>0,\qquad  &\forall \, x\in[0,2\pi), \label{p2f3}
\end{eqnarray}
are fulfilled.
}

If $f$ has a unique zero $x_{0}\in[0,2\pi)$, then
we set $P_{n}=C_{n}(p)$ where $p$ is a trigonometric
polynomial defined as
\begin{eqnarray}\label{polybis}
  p(x)=\prod_{\hat{x}\in\mathcal{M}_g(x_{0})}(2-2\cos(x-\hat{x}))^{\left\lceil \beta/2\right\rceil}\sim
  \prod_{\hat{x}\in\mathcal{M}_g(x_{0})}|x-\hat{x}|^{2\left\lceil \beta/2\right\rceil},
\end{eqnarray}
for $x\in[0,2\pi)$, with
\begin{eqnarray*}
  \beta\geq\beta_{\textrm{min}}=\min\left\{i\left|\lim_{x\rightarrow x_{0}}
  \frac{|x-x_{0}|^{2i}}{f(x)}<+\infty\right.\right\},
\end{eqnarray*}
thus conditions $(\ref{p2f1})$ and $(\ref{p2f3})$ are satisfied.

Before proving (in Section \ref{sect:TGMconv}) that conditions \eqref{p2f1} and \eqref{p2f3}
are sufficient to assure the TGM optimality, we consider
a crucial result both from a theoretical and a practical point of view.

\begin{proposition}\label{fhat} Let $f$ be a nonnegative function, $k=\frac{n}{g}\in\mathbb{N}$,
$p_{n,g}^{k}=C_{n}(p)Z_{n,g}^{k}\in\mathbb{C}^{n\times k}$,
with $p$ trigonometric polynomial satisfying the condition $(\ref{p2f1})$ for any
zero of $f$ and globally the condition $(\ref{p2f3})$. Then the matrix
$(p_{n,g}^{k})^{H}C_{n}(f)p_{n,g}^{k}\in\mathbb{C}^{k\times k}$
coincides with $C_{k}(\hat{f})$ where $\hat{f}$ is nonnegative and
\begin{eqnarray}\label{f2t}
  \hat{f}(x)=\frac{1}{g}\sum_{y\in\Omega_g\left(\frac{x}{g}\right)}f(y)|p|^{2}(y),
\end{eqnarray}
for $x\in[0,2\pi)$, i.e., the projected matrix is obtained picking every
$g$-th entry out of the symbol $f|p|^2$. In particular
\begin{itemize}
  \item[1.] if $f$ is a polynomial then $\hat{f}$ is a polynomial with a fixed degree $\left\lfloor \frac{c}{g}\right\rfloor$, where $c$ is the degree of $f|p|^2$;
  \item[2.] if $x_{0}$ is a zero of $f$ then $\hat{f}$ has a corresponding zero $y_{0}$ where $y_{0}=gx_{0}$ (mod $2\pi$);
  \item[3.] the order of the zero $y_{0}$ of $\hat{f}$ is exactly the same as the one of the zero $x_{0}$ of $f$, so that at the lower level the new projector can be easily defined in the same way.
\end{itemize}
\end{proposition}

\begin{proof} First we observe that, from $(\ref{cn})$ and $(\ref{deltan})$,
\begin{eqnarray*}
  (p_{n,g}^{k})^{H}C_{n}(f)p_{n,g}^{k}&=&(Z_{n,g}^{k})^{H}(C_{n}(p))^{H}C_{n}(f)C_{n}(p)Z_{n,g}^{k}\\
  &=&(Z_{n,g}^{k})^{H}C_{n}(f|p|^{2})Z_{n,g}^{k},
\end{eqnarray*}
thus the generating function of the circulant matrix
$(C_{n}(p))^{H}C_{n}(f)C_{n}(p)$ is $f|p|^{2}$.  Denoting by
$a_{j}$ the Fourier coefficients of $f|p|^{2}$, it holds
\begin{eqnarray*}
  C_{n}(f|p|^{2})=\left[a_{(r-s)\ {\rm mod}\,n}\right]_{r,s=0}^{n-1},
\end{eqnarray*}
and then, by $(\ref{Z})$, the entries of the matrix $(Z_{n,g}^{k})^{H}C_{n}(f|p|^{2})Z_{n,g}^{k}$ are given by
\begin{eqnarray*}
  \left[C_{n}(f|p|^{2})Z_{n,g}^{k}\right]_{r,s}&=&\sum_{\ell=0}^{n-1}[C_{n}(f|p|^{2})]_{r,\ell}[Z_{n,g}^{k}]_{\ell,s}\\
  &=&\sum_{\ell=0}^{n-1}a_{(r-\ell)\ {\rm mod}\,n}\delta_{\ell-gs}\\
  &=_{\textrm{(a)}}&a_{(r-gs)\ {\rm mod}\,n},\ \ \ r=0,\ldots,n-1,\ s=0,\ldots,k-1,\\
\end{eqnarray*}
\begin{eqnarray*}
  \left[(Z_{n,g}^{k})^{H}C_{n}(f|p|^{2})Z_{n,g}^{k}\right]_{r,s}&=&
  \sum_{\ell=0}^{n-1}[(Z_{n,g}^{k})^{H}]_{r,\ell}[C_{n}(f|p|^{2})Z_{n,g}^{k}]_{\ell,s}\\
  &=&\sum_{\ell=0}^{n-1}\delta_{\ell-gr}a_{(\ell-gs)\ {\rm mod}\,n}\\
  &=_{\textrm{(b)}}&a_{(gr-gs)\ {\rm mod}\,n}, \ \ \ r,s=0,\ldots,k-1,
\end{eqnarray*}
where (a) follows because there exists a unique $\ell\in{1,2,\ldots,n-1}$ such that $\ell-gs\equiv 0$ (mod $n$),
that is, $\ell\equiv gs$ (mod $n$) and, since $0\leq gs\leq n-1$, we obtain $\ell = gs$; similarly for (b).
Now if we denote by ${b}_{j}$ the Fourier coefficients of
$\hat{f}$ it only remains to show that ${b}_{(r-c)}=a_{(g r-g c)}$, $r,c=0,\ldots,k-1$, from which we directly infer that $[C_{n}(\hat{f})]_{r,c}={b}_{(r-c)\ {\rm mod}\,n}=a_{(g r-g c)\ {\rm mod}\,n}$. Since $f|p|^{2}$ is a polynomial, we can always write
\begin{eqnarray}\label{f2tbis}
  f|p|^{2}(x)=\sum_{\ell=-\infty}^{\infty}a_{\ell}e^{i\ell x},
  \qquad
 \hat{f}(x)=\sum_{\ell=-\infty}^{\infty}b_{\ell}e^{i\ell x}.
\end{eqnarray}

From $(\ref{coeff})$, $(\ref{f2t})$ and $(\ref{f2tbis})$, we have
\begin{eqnarray*}
  {b}_{r-c}&=&\frac{1}{2\pi}\int_{0}^{2\pi}\frac{1}{g}\sum_{j=0}^{g-1}
  \sum_{\ell=-\infty}^{+\infty}a_{\ell}e^{i\ell\left(\frac{x+2\pi j}{g}\right)}e^{-i(r-c)x}dx\\
  &=&\frac{1}{2\pi g}\int_{0}^{2\pi}\sum_{\ell=-\infty}^{+\infty}a_{\ell}\left(\sum_{j=0}^{g-1}e^{\frac{i2\pi \ell j}{g}}\right)
  e^{\frac{i\ell x}{g}}e^{-i(r-c)x}dx.
\end{eqnarray*}
Recalling that
\[
\frac{1}{g}\sum_{j=0}^{g-1}e^{\frac{i2\pi \ell j}{g}}=
\left\{\begin{array}{l@{\quad}l} 1 & \textrm{ if }\ell=g t\\ 0
&\textrm{otherwise} \end{array}\right.
\qquad
\textrm{and}
\qquad
\frac{1}{2\pi}\int_{0}^{2\pi} e^{i\ell x} dx=
\left\{\begin{array}{l@{\quad}l} 1 & \textrm{ if }\ell=0\\ 0
&\textrm{otherwise} \end{array}\right.,\]
we find that
\begin{eqnarray}\label{hata}
  \nonumber {b}_{r-c}&=&\frac{1}{2\pi g}\int_{0}^{2\pi}\sum_{t=-\infty}^{+\infty}a_{g t}
  g e^{\frac{ig t x}{g}}e^{-i(r-c)x}dx\\
  \nonumber&=&\sum_{t=-\infty}^{+\infty}a_{g t}\frac{1}{2\pi}\int_{0}^{2\pi}e^{ix(t-(r-c))}dx\\
  &=&a_{g(r-c)}.
\end{eqnarray}

From the expression of $\hat{f}$, since $f(x_{0})=0$ then it must hold
$p(y)=0$ $\forall y \in \mathcal{M}_g(x_{0})$ to satisfy the \eqref{p2f1}.
Thus $y_{0}=gx_{0}$ (mod $2\pi$) is a zero of $\hat{f}$ (i.e. item 2. is proved).

Moreover, by $(\ref{p2f3})$, we deduce that $p^{2}(x_{0})>0$
since $p^{2}(y)=0$, $\forall y \in \mathcal{M}_g(x_{0})$,
and the order of the zero $y_{0}$ of $f\left(\frac{x}{g}\right)|p|^{2}\left(\frac{x}{g}\right)$
is the same as the order of $f(x)$ at $x_{0}$.
Furthermore, by $(\ref{p2f1})$ we can see that $|p|^{2}\left(\frac{x+2\pi k}{g}\right)$ has at $y_{0}$
a zero of order at least equal to the one of $f(x)$ at $x_{0}$, for any $k=1,\ldots,g-1$.
Since all the contributions in $\hat{f}$ are nonnegative the thesis of item 3. follows.

Finally we have to prove item 1. Since ${b}_{j}$ are the Fourier coefficients of $\hat{f}$ and
$a_{j}$ are the Fourier coefficients of the polynomial $f|p|^{2}$, see $(\ref{f2tbis})$, from $(\ref{hata})$
we deduce that
\begin{eqnarray*}
  \hat{f}(x)=\sum_{j}{b}_{j}e^{ijx}=\sum_{j}a_{gj}e^{ijx}.
\end{eqnarray*}

Hence, if the polynomial $f|p|^{2}$ has degree $c$, $\hat{f}$ has degree at most $\left\lfloor \frac{c}{g}\right\rfloor$. $\hfill\square$
\end{proof}

\section{Proof of convergence}\label{proof:circ}
Using the results in Section \ref{cutting:circ}, it is possible to prove the optimality of
the TGM and of the W-cycle (W-cycle requires $g>2$).

\subsection{TGM convergence}\label{sect:TGMconv}
The smoothing property for $g=2$ was proved in \cite{mcirco}
and it holds unchanged also for $g>2$.

\begin{lemma}[\cite{mcirco}] \label{lm:smooth}
Let $A_{n}:=C_{n}(f)$ with $f$ being a nonnegative trigonometric polynomial
(not identically zero) and let $V_{n}:=I_{n}-\omega A_{n}$,
$0<\omega<2/\|f\|_{\infty}$.  If we choose $\alpha_{\rm{post}}$ so
that $\alpha_{\rm{post}}\leq a_{0}\omega(2-\omega\|f\|_{\infty})$
then relation $(\ref{cond1})$ holds true.
\end{lemma}
If in the previous Lemma we choose $\omega=\|f\|_{\infty}^{-1}$,
then $\alpha_{\rm{post}}\leq\|f\|_{1}/\|f\|_{\infty}$ and the best value of
$\alpha_{\rm{post}}$ is $\alpha_{\rm{post,best}}=\|f\|_{1}/\|f\|_{\infty}$.
Moreover, the result of Lemma \ref{lm:smooth} can be easily generalized when
considering both pre-smoothing and post-smoothing as in \cite{AD}.

The following result shows that TGM conditions \eqref{p2f1} and \eqref{p2f3}
are sufficient in order to satisfy the approximation property.

\begin{theorem} \label{th:tgmopt}
Let $A_{n}:=C_{n}(f)$ with $f$ being a nonnegative trigonometric polynomial
(not identically zero) and let $p_{n,g}^{k}=C_{n}(p)Z_{n,g}^{k}$ the
projector operator,  with $Z_{n,g}^{k}$ defined in $(\ref{Z})$ and
with $p$ trigonometric polynomial satisfying condition $(\ref{p2f1})$ for any
zero of $f$ and globally the condition
$(\ref{p2f3})$. Then, there exists a positive value $\gamma$
independent of $n$ such that inequality $(\ref{cond3})$ holds true.
\end{theorem}

\begin{proof} The proof is similar to that of Lemma 8.2 in \cite{Sun},
but we report it here for completeness.
First, we recall that the main diagonal
of $A_{n}$ is given by $D_{n}=a_{0}I_{n}$ with
$a_{0}=(2\pi)^{-1}\int_{Q}f=\|f\|_{1}>0$,  so that
$\|\cdot\|_{D_{n}}^{2}=a_{0}\|\cdot\|_{2}^{2}$.

In order to prove that there exists $\gamma>0$ independent of $n$
such that for any $x_{n}\in\mathbb{C}^{n}$
\begin{eqnarray*}
  \min_{y\in\mathbb{C}^{k}}\|x_{n}-p_{n,g}^{k}y\|_{D_{n}}^{2}=
  a_{0}\min_{y\in\mathbb{C}^{k}}\|x_{n}-p_{n,g}^{k}y\|_{2}^{2}
  \leq \gamma\|x_{n}\|_{A_{n}}^{2},
\end{eqnarray*}
we chose a special instance of $y$ in such a way that the
previous inequality  is reduced to a matrix inequality in the sense
of the partial ordering of the real space of the Hermitian matrices.
For any $x_{n}\in\mathbb{C}^{n}$, let
$\overline{y}\equiv\overline{y}(x_{n})\in\mathbb{C}^{k}$ be
defined as
\begin{eqnarray*}
  \overline{y}=\left[(p_{n,g}^{k})^{H}p_{n,g}^{k}\right]^{-1}(p_{n,g}^{k})^{H}x_{n}.
\end{eqnarray*}

Therefore, $(\ref{cond3})$ is implied by
\begin{eqnarray*}
  \|x_{n}-p_{n,g}^{k}\overline{y}\|_{2}^{2}\leq(\gamma/a_{0})\|x_{n}\|_{A_{n}}^{2},\qquad
  \forall x_{n}\in\mathbb{C}^{n},
\end{eqnarray*}
where the latter is equivalent to the matrix inequality
\begin{equation}\label{eq:W2}
  W_{n}(p)^{H}W_{n}(p)\leq(\gamma/a_{0})C_{n}(f),
\end{equation}
with $W_{n}(p)=I-p_{n,g}^{k}\left[(p_{n,g}^{k})^{H}p_{n,g}^{k}\right]^{-1}(p_{n,g}^{k})^{H}$.
Since, by construction, $W_{n}(p)$ is an Hermitian unitary projector,
it holds that $W_{n}(p)^{H}W_{n}(p)=W_{n}^{2}(p)=W_{n}(p)$.
As a consequence, inequality \eqref{eq:W2} can be rewritten as
\begin{eqnarray}\label{teocond1}
  W_{n}(p)\leq(\gamma/a_{0})C_{n}(f).
\end{eqnarray}

If $k=\frac{n}{g}\in\mathbb{N}$, following the decomposition in $(\ref{decompos})$,
$p_{n,g}^{k}=C_{n}(p)Z_{n,g}^{k}$ can be expressed according
to
\begin{eqnarray*}
  (p_{n,g}^{k})^{H}=\frac{1}{\sqrt{g}}F_{k}\left[\Delta_{p}^{(0)}|\Delta_{p}^{(1)}|\cdots|\Delta_{p}^{(g-1)}\right]F_{n}^{H},
\end{eqnarray*}
where
\begin{eqnarray*}
  \Delta_{p}^{(r)}= \begin{smallmatrix}\vspace{-0.5ex}\textrm{\normalsize diag}\\\vspace{-0.8ex}j=0,\ldots,k-1\end{smallmatrix}
  \left(p(x_{rk+j,n}^{(n)}) \right), \qquad r=0,\dots,g-1,
\end{eqnarray*}
with $x_{j}^{(n)}=2\pi j/n$.

Let $p[\mu] \in \mathbb{C}^{g}$ whose entries are given
by the evaluations of $p$ over the points of $\Omega(x_\mu^{(n)})$, for $\mu =0,\dots, k-1$.
There exists a suitable permutation by rows and columns of $F_{n}^{H}W_{n}(p)F_{n}$,
such that we can obtain a $g\times g$ block diagonal matrix whose
$\mu$th diagonal block is given by
$
I_g-p[\mu](p[\mu])^T/\|p[\mu]\|_2^2.
$
Therefore, using the same notation for $f[\mu]$ and denoting by ${\rm diag}(f[\mu])$
the diagonal matrix having the vector $f[\mu]$ on the main diagonal,
the condition $(\ref{teocond1})$ is equivalent to
\begin{equation}\label{teocond1block}
    I_g - \frac{p[\mu](p[\mu])^T}{\|p[\mu]\|_2^2} \leq (\gamma/a_{0}) {\rm diag}(f[\mu]),
\end{equation}
for $\mu =0,\dots, k-1$.
By the Sylvester inertia law \cite{GV}, the relation \eqref{teocond1block} is satisfied if
every entries of
\[
    {\rm diag}(f[\mu])^{-1/2}\left(I_g-\frac{p[\mu](p[\mu])^T}{\|p[\mu]\|_2^2} \right){\rm diag}(f[\mu])^{-1/2}
\]
is bounded in modulus by a constant, which follows from the TGM conditions \eqref{p2f1} and \eqref{p2f3}.

Furthermore, if we put
\begin{eqnarray*}
  z&=&\max_{y\in\Omega_g(x)}\left\|\frac{p^{2}(y)}{f(x)}\right\|_{\infty},\\
  h&=&\left\|\frac{1}{\sum_{y\in\Omega_g(x)}p^{2}(y)}\right\|_{\infty},
\end{eqnarray*}
the condition $(\ref{cond3})$  is satisfied choosing a value of $\gamma$ such that
$\gamma\geq g(g-1)a_{0}hz$. $\hfill\square$
\end{proof}

Combining Lemma \ref{lm:smooth} and Theorem \ref{th:tgmopt} with
Theorem \ref{teoconv}, it follows that the TGM convergence speed
does not depend on the size of the linear system.

\subsection{Multigrid convergence}

The optimal TGM convergence rate proved in Theorem \ref{th:tgmopt} can be extended
to a generic recursion level of the multigrid procedure
obtaining the so called ``level independency'' property.
The key tools to do that are the Proposition \ref{fhat} and an explicit
choice of the projector, considering for instance the symbol $p$ in \eqref{polybis}.
Indeed, the ``level independency'' was already proved in literature for $g=2$
(see \cite{CCS,Chan-Sun2,ADS}) and the proof can be extended to $g>2$, as 
in Theorem~\ref{th:tgmopt}.

The ``level independency'' implies that the W-cycle has a constant convergence rate
independent of the problem size \cite{Trot}.
However, the fact that the convergence speed does not depend on the size of the linear system
does not implies the optimality of the method, because the computational work
at each iteration is not taken into account. 
For estimating the computational work at each iteration of a multigrid method, we have to consider
the size of the coarse problem and the number $\theta$ of recursive calls.
In our case the size of the problem at the level $i$ is $n_i = gn_{i-1}$.
According to the analysis in \cite{Trot}, we assume that the multigrid components
(smoothing, projection, \dots)  require a number of arithmetic operations which is
$cn_i$, with $c$ constant independent of $n_i$, up to lower order term.
From equation $(2.4.14)$ in \cite{Trot}, the total computational work $C_m$
of one complete multigrid cycle is
\begin{equation}\label{eq:cost}
    C_m \doteq \left\{
    \begin{array}{l@{\qquad}l}
      \frac{g}{g-\theta}cn & \textrm{for }\, \theta<g \\
      O(n\log n) & \textrm{for } \, \theta=g
    \end{array}
    \right. \,,
\end{equation}
where the symbol $\doteq$ means equality up to lower order terms.
It follows that for $g=2$ the W-cycle can not be optimal  even in the presence of
``level independency'', because each multigrid iteration requires a computational cost
of $O(n\log n)$ while the matrix vector product is of $O(n)$. On the other hand,
for $g=3$ the W-cycle has $C_m\doteq3cn$ and hence it is optimal if the
``level independency'' is satisfied. More in general, the proposed multigrid
will be optimal for a number $\theta \in \mathbb{N}$ of recursive calls 
such that $1<\theta<g$.

\subsection{Some pathologies eliminated when using our algorithm}\label{pathology:circ}

From conditions $(3.4)$ and $(3.5)$ in \cite{mcirco}, we know that,
for $g=2$, if $x_{0}$ is a zero of $f$, then $f(x_{0}+\pi)$ must be
positive: otherwise relationship $(3.5)$ in \cite{mcirco} cannot be
satisfied with any polynomial $p$. But if we consider $g=3$ this is
no longer a problem, because conditions $(\ref{p2f1})$
 and $(\ref{p2f3})$ impose that, if $x_{0}$ is a zero
of $f$, then $f\left(x_{0}+\frac{2}{3}\pi\right)$ and $f\left(x_{0}+\frac{4}{3}\pi\right)$ must be
positive, while there are no conditions on $f(x_{0}+\pi)$.

For $g=3$, if $f$ has a unique zero $x_{0}\in[0,2\pi)$ of finite order,
then we consider $\hat{x}=x_{0}+\frac{2}{3}\pi$ (mod $2\pi$) and
$\tilde{x}=x_{0}+\frac{4}{3}\pi$ (mod $2\pi$) and we set
$P_{n}=C_{n}(p)$ where $p$ is a trigonometric polynomial defined as
\begin{eqnarray}\label{poly}
  &&\\
  &&\nonumber p(x)=(2-2\cos(x-\hat{x}))^{\left\lceil \beta/2\right\rceil}(2-2\cos(x-\tilde{x}))^{\left\lceil \beta/2\right\rceil}\sim
  |x-\hat{x}|^{2\left\lceil \beta/2\right\rceil}|x-\tilde{x}|^{2\left\lceil \beta/2\right\rceil},
\end{eqnarray}
for $x\in[0,2\pi)$, with
\begin{eqnarray*}
  \beta\geq\beta_{\textrm{min}}=\min\left\{i\left|\lim_{x\rightarrow x_{0}}
  \frac{|x-x_{0}|^{2i}}{f(x)}<+\infty\right.\right\},
\end{eqnarray*}
thus conditions $(\ref{p2f1})$  and $(\ref{p2f3})$ are
satisfied.  If $f$ shows more than one zero in $[0,2\pi)$, then we
consider a polynomial $p$ which is the product of the basic
polynomials of kind $(\ref{poly})$, satisfying the condition
$(\ref{p2f1})$ for any single zero and globally
the condition \nolinebreak $(\ref{p2f3})$.

\begin{example}\label{ex:double}
The symbol
\begin{eqnarray*}
  f(x)=(2-2\cos(x))(2+2\cos(x))
\end{eqnarray*}
vanishes at zero and $\pi$ with order two.
For $g=3$, we have
$\mathcal{M}_3(0) = \{\frac{2\pi}{3},\,\frac{4\pi}{3}\}$ and
$\mathcal{M}_3(\pi) = \{\frac{5\pi}{3},\,\frac{7\pi}{3}\}$,
thus the trigonometric polynomial
\begin{equation}\label{eq:pdouble}
  p(x)=\prod_{\hat{x} \in \mathcal{M}_3(0) \bigcup \mathcal{M}_3(\pi)}(2-2\cos(x-\hat{x}))
\end{equation}
satisfies the TGM conditions \eqref{p2f1} and \eqref{p2f3}
and defines an optimal TGM.
\end{example}

\section{Numerical experiments}\label{sec:num}

In this section, we apply the proposed multigrid method to symmetric positive definite
circulant and Toeplitz systems $A_{n}x=b$.
We choose as solution the vector $x$ such that $x_{i}=i/n$, $i=1,\ldots,n$.
The right-hand side vector $b$ is obtained accordingly.
As smoother, we use Richardson with $\omega_{j}=1/\|f_{j}\|_\infty$, for $j=0,\ldots,m-1$ ($m$ is number of subgrids in the algorithm, $m=1$ for the TGM), for pre-smoother and the conjugate gradient for post-smoother. In the V-cycle and W-cycle procedure when the coarse grid size is less than or equal to 27, we solve the coarse grid system exactly. The zero vector is used as the initial guess and the stopping criterion is $\|r_{q}\|_2/\|r_{0}\|_2\leq 10^{-7}$, where $r_{q}$ is the residual vector after $q$ iterations and $10^{-7}$ is the given tolerance.

\subsection{Cutting operators for Toeplitz matrices}
When dealing with circulant matrices, using the projector defined in $(\ref{procirc})$, 
the matrix at the lower level is still a circulant matrix, while for Toeplitz matrices, if we consider $A_{n}:=T_{n}(f)$ and $p_{n,3}^{k}=T_{n}(p)Z_{n,3}^{k}$, where $p$ is defined in accordance with the formula $(\ref{poly})$ and $k=\frac{n}{3}\in\mathbb{N}$, we find that
\begin{eqnarray*}
  T_{n}(p)T_{n}(f)T_{n}(p)=T_{n}(fp^{2})+G_{n}(f,p).
\end{eqnarray*}
Furthermore, if $2\beta+1$ is the bandwidth of $T_{n}(p)$, the matrix $G_{n}(f,p)$ has rank $2\beta$ and is formed
by a matrix of rank $\beta$ in the upper left corner and a matrix of the same rank in the bottom right corner. 
According to the proposal in \cite{ADS},
 we take a cutting matrix that will completely erase the contribution of $G_{n}(f,p)$, so that, at the lower level, the restriction of the matrix $T_{n}(p)T_{n}(f)T_{n}(p)$ is still a Toeplitz matrix and thus we can recursively apply the algorithm. The proposed cutting matrix is as follows:
\begin{eqnarray*}
  \widetilde{Z}_{n,3}^{k}=\left[\begin{array}{c}
  0_{\beta}^{k-r}\\
  Z_{n-2\beta,3}^{k-r}\\
  0_{\beta}^{k-r}\\
  \end{array}\right]_{n\times k-r}\qquad r=\frac{2(\beta-1)}{3},
\end{eqnarray*}
where $0_{\beta}^{k-r}$ is the zero matrix of size $\beta\times(k-r)$; $\widetilde{Z}_{n,3}^{k}$ has the first and the last $\beta$ rows equal to zero and therefore it is able to remove corrections of rank less than or equal to $2\beta$. Since $G_{n}(f,p)$ has rank $2\beta$, we deduce that $A_{k-r}=(\widetilde{Z}_{n,3}^{k})^{H}T_{n}(p)T_{n}(f)T_{n}(p)\widetilde{Z}_{n,3}^{k}$ is Toeplitz and we cannot obtain a Toeplitz matrix of size greater than this. As a consequence, for Toeplitz matrices, the projector is then defined as
\begin{eqnarray*}
  p_{n,3}^{k}=T_{n}(p)\widetilde{Z}_{n,3}^{k}.
\end{eqnarray*}

Also the size of the problem should be chosen in such a way that a recursive application of the algorithm
is possible; in our case, if we choose $n=3^{\alpha}-\xi$ with $\xi=\beta-1$, then the size of the problem at the lower level becomes $k'=k-r=\frac{n-2(\beta-1)}{3}=\frac{3^{\alpha}-(\beta-1)-2(\beta-1)}{3}=3^{\alpha-1}-(\beta-1)=3^{\alpha-1}-\xi$.

\subsection{Zero at the origin and at $\pi$.}
We present some examples where the generating functions $f_{0}$ vanish at the origin and at $\pi$.
Firstly, we consider the Example \ref{ex:double} where the symbol
\begin{eqnarray*}
  f_{0}(x)=(2-2\cos(x))(2+2\cos(x)),
\end{eqnarray*}
vanishes at $0$ and $\pi$ with order $2$.
According to $(\ref{poly})$, we choose the projector $p_{n,3}^{k}=C_{n}(p_{0})Z_{n,3}^{k}$ if $A_{n}$ is a circulant matrix and $p_{n,3}^{k}=T_{n}(p_{0})\widetilde{Z}_{n,3}^{k}$ if $A_{n}$ is a Toeplitz matrix, where
$p_0=p$ defined in \eqref{eq:pdouble}.
Fixing $x_{0}^{(1)}=0$ and $x_{0}^{(2)}=\pi$, the position of the new zeros $x_{k}^{(j)}$, for $k=1,2,\ldots,$ $m-1$ with $j=1,2$, move according to Proposition \ref{fhat} and, in this case, the functions $p_{k}$ are equal to $p$ for every level $k$.
Tables \ref{tab1} and \ref{tab2} report the number of iterations required for convergence in the case of circulant and Toeplitz systems, respectively. In all cases we note an optimal behavior at exception of the V-cycle for Toeplitz matrices where the number of iterations slightly grows with the size $n$.
\begin{table}
\caption{Circulant case: $f_{0}(x)=(2-2\cos(x))(2+2\cos(x))$.}\label{tab1}
\begin{center}
\begin{tabular}{lcccccc}
\hline\noalign{\smallskip}
 n & \multicolumn{6}{c}{$\#$ iterations}\\
\hline
  & \multicolumn{2}{c}{Two-grid} & \multicolumn{2}{c}{V-cycle} & \multicolumn{2}{c}{W-cycle}\\
  & $\nu_{\rm{pre}}=$ & $\nu_{\rm{pre}}=$ & $\nu_{\rm{pre}}=$ & $\nu_{\rm{pre}}=$ & $\nu_{\rm{pre}}=$ & $\nu_{\rm{pre}}=$\\
  & $\nu_{\rm{post}}=1$ & $\nu_{\rm{post}}=2$ & $\nu_{\rm{post}}=1$ & $\nu_{\rm{post}}=2$ & $\nu_{\rm{post}}=1$ & $\nu_{\rm{post}}=2$\\
\noalign{\smallskip}\hline\noalign{\smallskip}
$3^{4}=81$ & 11 & 6 & 11 & 6 & 11 & 6\\
$3^{5}=243$ & 11 & 6 & 11 & 7 & 11 & 6\\
$3^{6}=729$ & 11 & 6 & 11 & 7 & 11 & 6\\
$3^{7}=2187$ & 11 & 6 & 11 & 7 & 11 & 6\\
\noalign{\smallskip}\hline
\end{tabular}
\end{center}
\end{table}
\begin{table}
\caption{Toeplitz case. $f_{0}(x)=(2-2\cos(x))(2+2\cos(x))$}\label{tab2}
\begin{center}
\begin{tabular}{lcccccc}
\hline\noalign{\smallskip}
 n & \multicolumn{6}{c}{$\#$ iterations}\\
\hline
  & \multicolumn{2}{c}{Two-grid} & \multicolumn{2}{c}{V-cycle} & \multicolumn{2}{c}{W-cycle}\\
  & $\nu_{\rm{pre}}=$ & $\nu_{\rm{pre}}=$ & $\nu_{\rm{pre}}=$ & $\nu_{\rm{pre}}=$ & $\nu_{\rm{pre}}=$ & $\nu_{\rm{pre}}=$\\
  & $\nu_{\rm{post}}=1$ & $\nu_{\rm{post}}=2$ & $\nu_{\rm{post}}=1$ & $\nu_{\rm{post}}=2$ & $\nu_{\rm{post}}=1$ & $\nu_{\rm{post}}=2$\\
\noalign{\smallskip}\hline\noalign{\smallskip}
$3^{4}-3=78$ & 24 & 14 & 24 & 14 & 24 & 14\\
$3^{5}-3=240$ & 24 & 15 & 35 & 20 & 28 & 16\\
$3^{6}-3=726$ & 24 & 15 & 43 & 24 & 29 & 16\\
$3^{7}-3=2184$ & 24 & 15 & 49 & 27 & 29 & 16\\
\noalign{\smallskip}\hline
\end{tabular}
\end{center}
\end{table}

In the second example we increase the order of the zero in $\pi$ considering the function
\[  f_{0}(x)=(2-2\cos(x))(2+2\cos(x))^{2}.\]
which has a zero at 0 with order 2 and one at $\pi$ with order 4.
The polynomial $p_{0}=p$ defined in \eqref{eq:pdouble} still satisfies the TGM conditions \eqref{p2f1} and  \eqref{p2f3}.
The functions $p_{k}$ do not change at the lower levels.
In Tables \ref{tab3} and \ref{tab4} we report the number of iterations required for convergence in the case of circulant and Toeplitz systems, respectively.
Since $f_0$ has a zero of order 4 the condition number of $A_{3^7} = O((3^7)^4) = O(10^{13})$.
Therefore, using double precision, for this example we choose a tolerance equal to $10^{-3}$.
This choice agrees also with the plots in Figures \ref{im1} and \ref{im2} where we note an optimal reduction of the
residual norm only until about $10^{-3}$.

\begin{table}
\caption{Circulant case. $f_{0}(x)=(2-2\cos(x))(2+2\cos(x))^2$, tolerance=$10^{-3}$}\label{tab3}
\begin{center}
\begin{tabular}{lcccccc}
\hline\noalign{\smallskip}
 n & \multicolumn{6}{c}{$\#$ iterations}\\
\hline
  & \multicolumn{2}{c}{Two-grid} & \multicolumn{2}{c}{V-cycle} & \multicolumn{2}{c}{W-cycle}\\
  & $\nu_{\rm{pre}}=$ & $\nu_{\rm{pre}}=$ & $\nu_{\rm{pre}}=$ & $\nu_{\rm{pre}}=$ & $\nu_{\rm{pre}}=$ & $\nu_{\rm{pre}}=$\\
  & $\nu_{\rm{post}}=1$ & $\nu_{\rm{post}}=2$ & $\nu_{\rm{post}}=1$ & $\nu_{\rm{post}}=2$ & $\nu_{\rm{post}}=1$ & $\nu_{\rm{post}}=2$\\
\noalign{\smallskip}\hline\noalign{\smallskip}
$3^{4}=81$   & 20 & 9 & 20 & 9   & 20  & 9 \\
$3^{5}=243$  & 20 & 9 & 18 & 9   & 20  & 9 \\
$3^{6}=729$  & 20 & 9 & 18 & 9   & 20  & 9 \\
$3^{7}=2187$ & 20 & 9 & 18 & 9   & 20  & 9 \\
\noalign{\smallskip}\hline
\end{tabular}
\end{center}
\end{table}
\begin{table}
\caption{Toeplitz case. $f_{0}(x)=(2-2\cos(x))(2+2\cos(x))^2$, tolerance=$10^{-3}$}\label{tab4}
\begin{center}
\begin{tabular}{lcccccc}
\hline\noalign{\smallskip}
 n & \multicolumn{6}{c}{$\#$ iterations}\\
\hline
  & \multicolumn{2}{c}{Two-grid} & \multicolumn{2}{c}{V-cycle} & \multicolumn{2}{c}{W-cycle}\\
  & $\nu_{\rm{pre}}=$ & $\nu_{\rm{pre}}=$ & $\nu_{\rm{pre}}=$ & $\nu_{\rm{pre}}=$ & $\nu_{\rm{pre}}=$ & $\nu_{\rm{pre}}=$\\
  & $\nu_{\rm{post}}=1$ & $\nu_{\rm{post}}=2$ & $\nu_{\rm{post}}=1$ & $\nu_{\rm{post}}=2$ & $\nu_{\rm{post}}=1$ & $\nu_{\rm{post}}=2$\\
\noalign{\smallskip}\hline\noalign{\smallskip}
$3^{4}-3=78$   & 50 & 31 &  50    &  31     & 50  & 31  \\
$3^{5}-3=240$  & 48 & 31 &  93    &  35     & 72  & 32  \\
$3^{6}-3=726$  & 47 & 31 &  74    &  34     & 68  & 31  \\
$3^{7}-3=2184$ & 47 & 31 &  76    &  34     & 68  & 31  \\
\noalign{\smallskip}\hline
\end{tabular}
\end{center}
\end{table}

\begin{figure}
    \centering
    \includegraphics[width=0.49\textwidth]{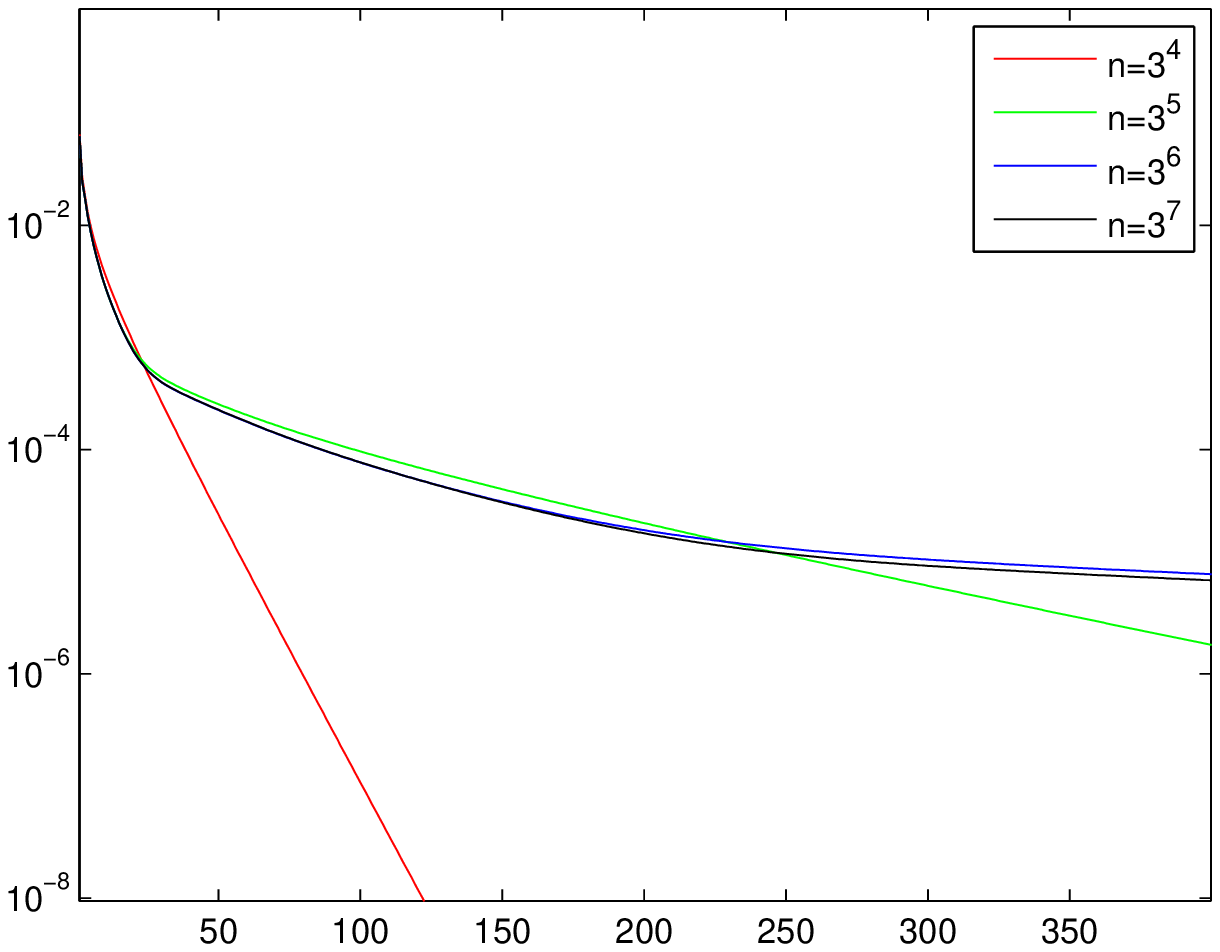}
    \includegraphics[width=0.49\textwidth]{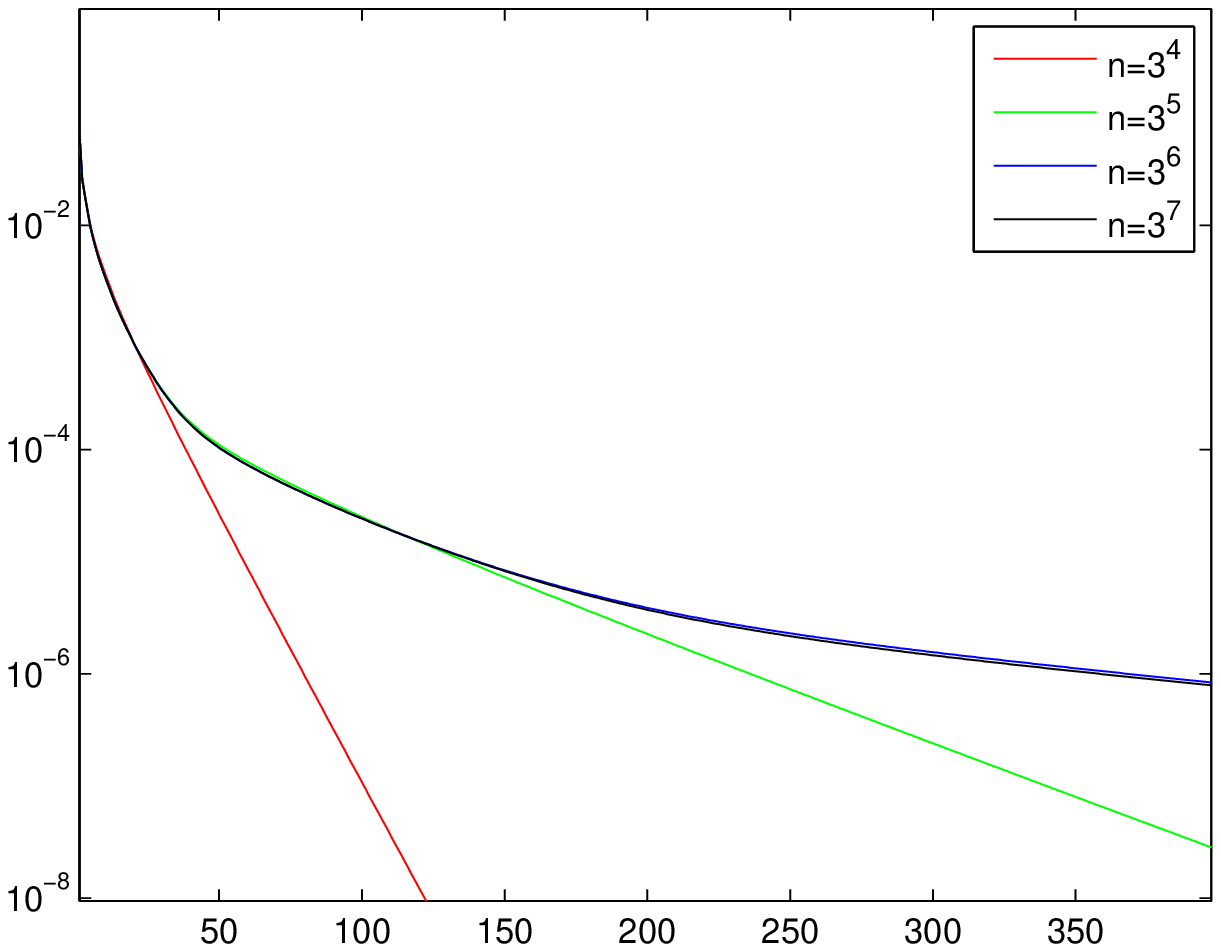}
    \caption{Circulant: Graph of the residual in logarithmic scale of the V-cycle (left) and W-cycle (right) with different sizes $n$, with $\nu_{\rm{pre}}=\nu_{\rm{post}}=1$ and a fixed number of iterations $iter=400$; $f_{0}(x)=(2-2\cos(x))(2+2\cos(x))^2$.}\label{im1}
\end{figure}
\begin{figure}
    \centering
    \includegraphics[width=0.49\textwidth]{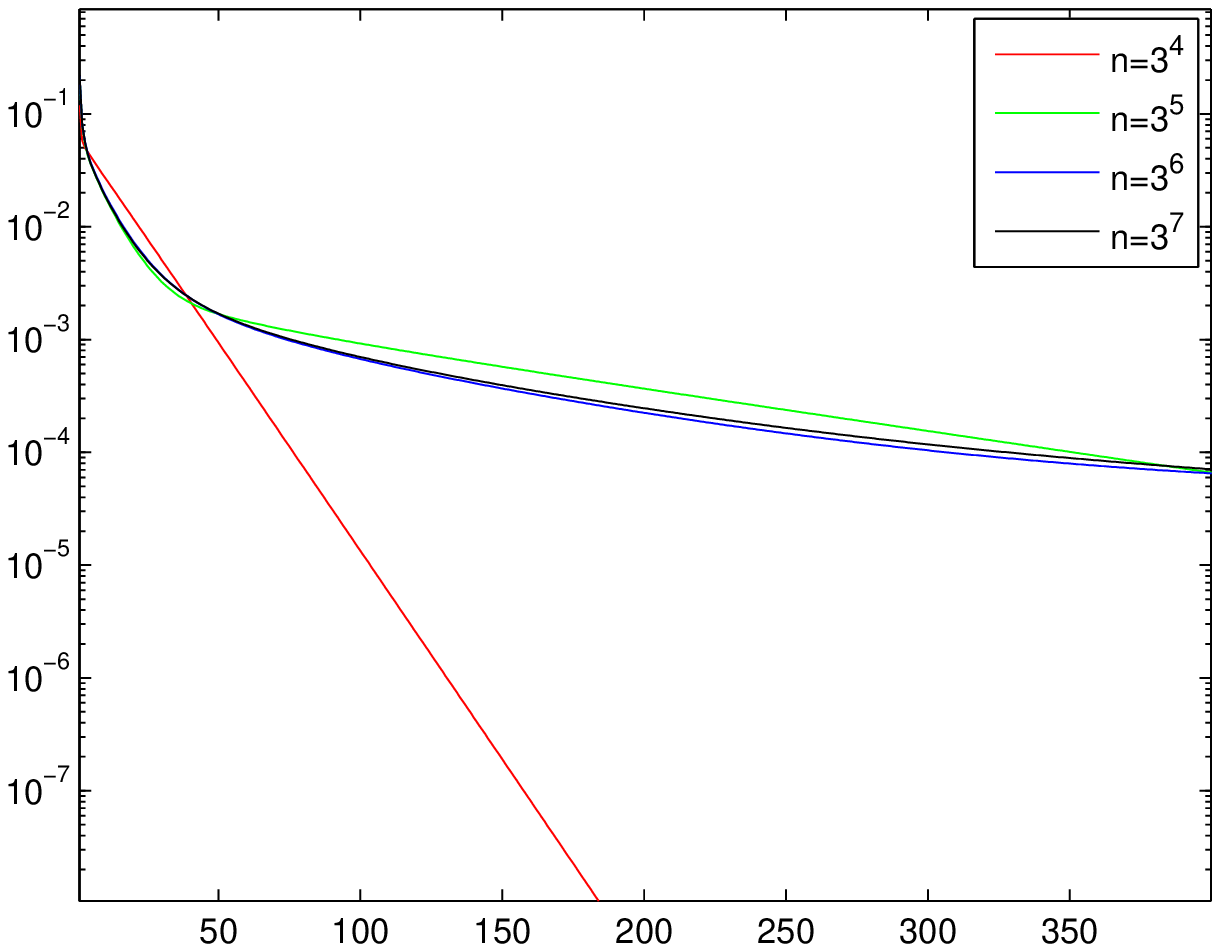}
    \includegraphics[width=0.49\textwidth]{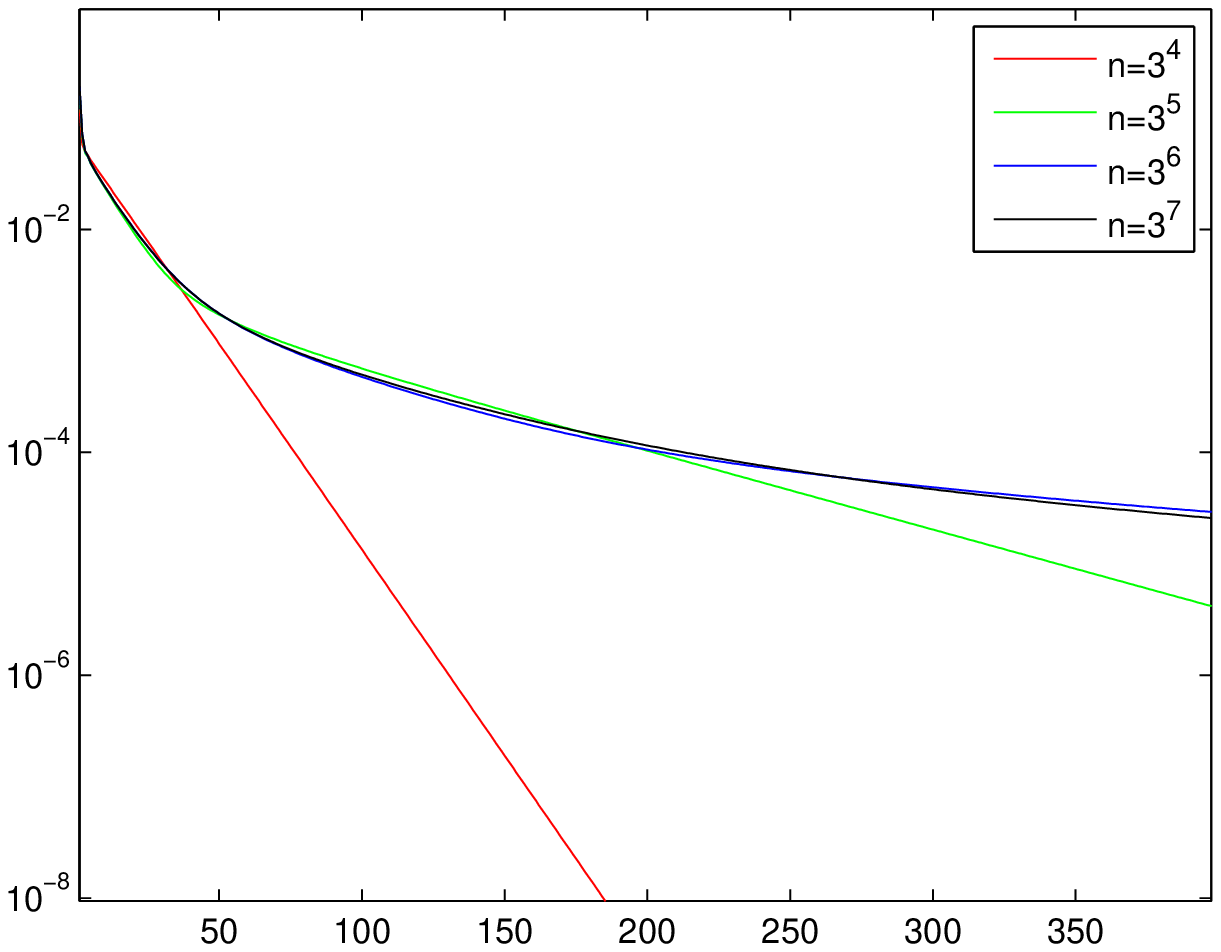}
    \caption{Toeplitz: Graph of the residual in logarithmic scale of the V-cycle (left) and W-cycle (right) with different sizes $n$, with $\nu_{\rm{pre}}=\nu_{\rm{post}}=1$ and a fixed number of iterations $iter=400$; $f_{0}(x)=(2-2\cos(x))(2+2\cos(x))^2$.}\label{im2}
\end{figure}

The last example of this subsection is taken from \cite{CCS}.
The generating function
\[
  f_{0}(x)=6-4\cos(2x)-2\cos(4x),
\]
vanishes at $0$ and $\pi$ with order $2$.
The symbol of the projector is again $p_0=p$ defined in \eqref{eq:pdouble}.
The initial guess is a random vector $u$ such that $0\leq u_{j}\leq 1$,
the pre-smoother is a step of damped Jacobi with parameter $\omega_{j}=[A_j]_{1,1}/\|f(x)\|_\infty$
while the post-smoother is a step of damped Jacobi with parameter $\omega_{j}=2[A_j]_{1,1}/\|f(x)\|_\infty$
for $j=0,\dots,m-1$.
The coarser problem is fixed such that is has size lower than 6.
Table  \ref{tabIV} shows that the number of iterations required to achieve the tolerance $10^{-7}$ remains constant increasing the
size $n$ of the system like for the multigrid technique proposed in \cite{CCS}. The number of iterations is reasonable
in both cases even if a direct comparison can not be done because of the difference in the choice of the projection techniques and in the size of the projected problems.
\begin{table}
\caption{Toeplitz case. $f_{0}(x)=6-4\cos(2x)-2\cos(4x)$, $\nu_{\rm{pre}}=\nu_{\rm{post}}=1$, tolerance=$10^{-7}$}\label{tabIV}
\begin{center}
\begin{tabular}{lccc}
\hline\noalign{\smallskip}
 n & \multicolumn{3}{c}{$\#$ iterations}\\
\hline
  & Two-grid            & W-cycle               & V-cycle\\
\noalign{\smallskip}\hline\noalign{\smallskip}
$3^{4}-3=78$   & 15 & 19 & 28  \\
$3^{5}-3=240$  & 15 & 20 & 39  \\
$3^{6}-3=726$  & 14 & 20 & 45  \\
$3^{7}-3=2184$ & 13 & 20 & 47  \\
\noalign{\smallskip}\hline
\end{tabular}
\end{center}
\end{table}

\subsection{Some Toeplitz examples}
In this subsection we consider only the more interesting case
for practical applications: Toeplitz matrices with a multigrid strategy.

The first example is a function with a zero not at the origin or $\pi$:
\[  f_{0}(x)=\left(2-2\cos\left(x-\frac{\pi}{3}\right)\right)\]
which vanishes at $\pi/3$ with order 2.
Moreover, we choose as true solution a random vector instead of a smooth solution.
The tolerance is again $10^{-7}$.
The symbol of the projector at the first level is
\begin{eqnarray*}
  p_{0}(x)=\left(2-2\cos\left(x-\pi\right)\right)\left(2-2\cos\left(x-\frac{5}{3}\pi\right)\right),
\end{eqnarray*}
while at the lower levels it changes with the zero of $f_j$
which moves according to Proposition \ref{fhat}.
Table \ref{tabIII} shows an optimal convergence both for V-cycle and W-cycle.

\begin{table}
\caption{Toeplitz case. $f_{0}(x)=\left(2-2\cos\left(x-\frac{\pi}{3}\right)\right)$, tolerance=$10^{-7}$.}\label{tabIII}
\begin{center}
\begin{tabular}{lcccc}
\hline\noalign{\smallskip}
 n & \multicolumn{4}{c}{$\#$ iterations}\\
\hline
  & \multicolumn{2}{c}{V-cycle} & \multicolumn{2}{c}{W-cycle}\\
  & $\nu_{\rm{pre}}=$ & $\nu_{\rm{pre}}=$ & $\nu_{\rm{pre}}=$ & $\nu_{\rm{pre}}=$\\
  & $\nu_{\rm{post}}=1$ & $\nu_{\rm{post}}=2$ & $\nu_{\rm{post}}=1$ & $\nu_{\rm{post}}=2$\\
\noalign{\smallskip}\hline\noalign{\smallskip}
$3^{4}-1=80$   & 33 & 37  & 33 & 37\\
$3^{5}-1=242$  & 30 & 31  & 30 & 31\\
$3^{6}-1=728$  & 30 & 31  & 30 & 31\\
$3^{7}-1=2186$ & 30 & 31  & 30 & 31\\
\noalign{\smallskip}\hline
\end{tabular}
\end{center}
\end{table}

In the second example, we consider the dense Toeplitz matrix generated by the function
$f(x)=x^{2}$,
which has the Fourier series expansion
\begin{eqnarray*}
  f(x)=\frac{\pi^{2}}{3}-4\left(\frac{\cos(x)}{1^{2}}-\frac{\cos(2x)}{2^{2}}+\frac{\cos(3x)}{3^{2}}-\cdots\right).
\end{eqnarray*}
Such function shows a unique zero at $0$ with order $2$ and hence we use the projector with symbol
\begin{eqnarray*}
  p_{0}(x)=\left(2-2\cos\left(x-\frac{2}{3}\pi\right)\right)\left(2-2\cos\left(x-\frac{4}{3}\pi\right)\right).
\end{eqnarray*}
In Table \ref{tab5} we report the number of iterations required for the convergence with the preassigned accuracy
and we note again the optimal behavior.
\begin{table}
\caption{Toeplitz case. $f(x)=x^2$}\label{tab5}
\begin{center}
\begin{tabular}{lcccc}
\hline\noalign{\smallskip}
 n & \multicolumn{4}{c}{$\#$ iterations}\\
\hline
  & \multicolumn{2}{c}{V-cycle} & \multicolumn{2}{c}{W-cycle}\\
  & $\nu_{\rm{pre}}=$ & $\nu_{\rm{pre}}=$ & $\nu_{\rm{pre}}=$ & $\nu_{\rm{pre}}=$\\
  & $\nu_{\rm{post}}=1$ & $\nu_{\rm{post}}=2$ & $\nu_{\rm{post}}=1$ & $\nu_{\rm{post}}=2$\\
\noalign{\smallskip}\hline\noalign{\smallskip}
$3^{4}-1=80$  & 21 & 11 & 21 & 11 \\
$3^{5}-1=242$ & 18 & 11 & 21 & 11 \\
$3^{6}-1=728$ & 18 & 11 & 21 & 11 \\
$3^{7}-1=2186$ & 18 & 11 & 21 & 11 \\
\noalign{\smallskip}\hline
\end{tabular}
\end{center}
\end{table}

\section{Conclusions and future work}\label{sec:final}

In this paper we have extended the rigorous two-grid analysis for
circulant matrices to the case where the size reduction is performed by a factor $g$
with $g>2$. The interesting novelty is that the new size reduction strategy allows
to eliminate some pathologies which occur when $g=2$. In particular,
if the considered matrices come from the approximation of certain integro-differential equations
then we have two source of ill-conditioning and the zeros of the underlying symbol are located
at zero and at $\pi$: this situation is a special case of mirror point zeros
and, when $g=2$, it is possible to prove that the resulting two-grid iteration cannot be optimal
(see \cite{FS1,Sun}).
Such difficulty can be overcome when we choose a larger $g$.
Moreover, when increasing $g$ the size of the coarse problems decreases: as a consequence
more multigrid recursive calls could be considered, like the
W-cycle which is proved to be optimal for $g \geq 3$.

We stress that  the numerical experiments are encouraging not only for circulant matrices
but also regarding Toeplitz matrices and concerning the V-cycle algorithm. A
future line of research must include the
multilevel setting, following the approach in \cite{Sun,AD}, and a rigorous proof of convergence for the whole
V-cycle procedure in accordance with the proof technique introduced in \cite{ADS}.


\end{document}